\documentclass{amsart}
\usepackage{amssymb,amsfonts,amsmath}
\usepackage[all]{xy}
\usepackage[shortlabels]{enumitem}
\pdfoutput=1
\usepackage{commath}
\usepackage{esdiff}	
\usepackage{graphicx,subfigure}	
\usepackage[center]{caption}
\usepackage[sectionbib]{natbib}
\usepackage[colorlinks,
            linkcolor=blue,      
            anchorcolor=blue, 
            citecolor=blue,]{hyperref}

\theoremstyle{theorem}
\newtheorem{theorem}{Theorem}[section]
\newtheorem{lemma}[theorem]{Lemma}
\newtheorem{proposition}[theorem]{Proposition}
\newtheorem{corollary}[theorem]{Corollary}

\theoremstyle{definition}
\newtheorem{definition}[theorem]{Definition}
\newtheorem{example}[theorem]{Example}
\newtheorem{remark}[theorem]{Remark}

\numberwithin{equation}{section}
\numberwithin{figure}{section}

\begin{document}

\title{Non-positively curved Ricci Surfaces with catenoidal ends}

\author{Yiming ZANG}

\address{Université de Lorraine, CNRS, IECL, F-54000 Nancy, France}

\email{yiming.zang@univ-lorraine.fr}

\maketitle

\begin{abstract}

A Ricci surface is defined as a Riemannian surface $(M,g_M)$ whose Gauss curvature satisfies the differential equation $K\Delta K + g_M\big(dK,dK\big) + 4K^3=0$. Andrei Moroianu and Sergiu Moroianu proved that a Ricci surface with non-positive Gauss curvature admits locally a minimal immersion into $\mathbb{R}^3$. In this paper, we are interested in studying non-compact orientable Ricci surfaces with catenoidal ends. We use an analogue of the Weierstrass data to obtain some classification results for such Ricci surfaces. We also give an existence result for positive genus Ricci surfaces with catenoidal ends. 

\end{abstract}

\footnote{\text{MSC2021:} 53A10, 53C20.}   
\footnote{\text{Keywords:} Ricci surfaces; minimal surfaces; catenoidal ends; Weierstrass data.}

\section{Introduction}

A classical question in the theory of minimal surfaces is to study when a Riemannian surface $(M,g_M)$ can be locally isometrically immersed into $\mathbb{R}^3$ as a minimal surface. In the history, the first result for this question was given by Gregorio Ricci-Curbastro. He provided a necessary and sufficient condition for the existence of such immersions near points with negative Gauss curvature (see \citep{ricci1895sulla}). Andrei Moroianu and Sergiu Moroianu have proven later in \citep{moroianu2015ricci} that the assumption of negative Gauss curvature could actually be left out. The main idea of their proof is to study the differential equation
\begin{align*}
        K\Delta K + g_M\big(dK,dK\big) + 4K^3=0
        \end{align*}
satisfied by the Gauss curvature appearing in Ricci's theorem. It is often called the Ricci condition in the case $K<0$. A Riemannian surface whose metric satisfies this equation is called a Ricci surface. 

The theory of compact Ricci surfaces is developed by Andrei Moroianu and Sergiu Moroianu as well. To be more precise, they provided some methods in constructing compact Ricci surfaces. We are therefore interested in considering non-compact orientable Ricci surfaces. In 1958, Huber proved that a complete non-positively curved Riemannian surfaces with finite total curvature has to be biholomorphic to a compact Riemann surface with finite number of punctured points (see \citep{huber1958subharmonic}). Consequently, a problem which should be naturally posed is to determine all the possible Ricci metrics on a given Riemann surface. However, this project seems to be very ambitious because many cases may occur around punctured points. The goal of our work is to classify Ricci surfaces by adding an assumption on the punctured points which is called the condition for catenoidal ends. 

In this paper, Section 2 is devoted to the recall of some basic properties of the well-known Weierstrass representation for minimal immersions in $\mathbb{R}^3$. The definition of Ricci surfaces is introduced in Section 3, then we focus on defining the Weierstrass data for Ricci surfaces which will be our principal tools. In Section 4, we plan to define catenoidal ends for Ricci surfaces, and we will prove a lemma providing a local description of the Weierstrass data near a catenoidal end. 

Our first main result is theorem \ref{thm4.1} which will be proven in Section 5. This theorem offers a complete classification of non-positively curved Ricci surfaces $M\simeq\mathbb{C}\setminus\lbrace0\rbrace$ with two catenoidal ends. It assures that a such Ricci surface can only be isometric to one of the two classes of surfaces. We will give also an explanation for the relations between our theorem and some related results. For example, theorem \ref{thm4.1} is equivalent to a result of Troyanov which classified all the metrics of constant Gauss curvature 1 on $\bar{\mathbb{C}}$ with two conical singularities at $0$ and $\infty$ (see \citep{troyanov1989metrics}).

In Section 6, we are going to discuss the case when the Ricci surface is biholomorphic to $M\simeq\mathbb{C}\setminus\lbrace0,1\rbrace$ possessing three catenoidal ends. Since things become much more complicated in this situation, it could be very hard to obtain a complete classification as in the previous case. We restrict ourselves to several classification results under some additional conditions on the total curvature and the reducibility of induced metric $(-K)g_M$. For concrete statements of these results, readers may consult theorem \ref{thm6.1}, theorem \ref{thm6.2} and theorem \ref{thm6.3}.
 
In the last section of this paper, we extend our discussion to Ricci surfaces with positive genus. We have proven theorem \ref{thm7.2} which assures that for $k,n>0$, there exists a Ricci surface with genus $k$ and $n$ catenoidal ends. 

The author wishes to express his gratitude to his supervisor, Prof. Benoît Daniel, for his helpful discussions and encouragement.

\section{Preliminaries}

In this section, we plan to recall the Weierstrass representation for minimal immersions in $\mathbb{R}^3$ which plays an essential role in our discussion.

Suppose that $X:M\rightarrow\mathbb{R}^3$ is a minimal immersion, where $M$ is a smooth manifold of dimension 2. Taking a simply connected neighbourhood $(U,(x,y))$ of $p_0\in M$ with $(x,y)$ the isothermal coordinates, then $X$ being minimal means that $X$ is a harmonic map, or equivalently,
\begin{equation} \label{1.1}
\phi=(\phi_1,\phi_2,\phi_3):=\frac{\partial X}{\partial x}-i\frac{\partial X}{\partial y}=2\frac{\partial X}{\partial z}
\end{equation}
is holomorphic, where $\frac{\partial}{\partial z}=\frac{1}{2}(\frac{\partial}{\partial x}-i\frac{\partial}{\partial y})$ and $\frac{\partial}{\partial \bar{z}}=\frac{1}{2}(\frac{\partial}{\partial x}+i\frac{\partial}{\partial y})$. 
It is not difficult to verify that $\phi dz$ is a globally defined vector valued holomorphic 1-form on $M$. With the help of $\phi dz$, we may construct a meromorphic function $g$ and a holomorphic 1-form $\eta$ on $M$ such that the minimal immersion $X$ can be expressed as 
\begin{equation}\label{1.2}
X(p)=X(p_0)+\textrm{Re}\int_{p_0}^{p}\left(\frac{1}{2}\left(1-g^2\right)\eta,\frac{i}{2}\left(1+g^2\right)\eta,g\eta\right).
\end{equation}
Formula (\ref{1.2}) is called the Weierstrass representation of the minimal immersion $X:M\rightarrow\mathbb{R}^3$ and the pair $(g,\eta)$ is called the Weierstrass data of $X$. Since $\phi dz$ is a holomorphic 1-form which does not vanish, we see that whenever $g$ has a pole of order $m$ at $p\in M$, $\eta$ must have a zero of order $2m$ at $p\in M$. The function $g$ is also called the Gauss map. To see this, it is sufficient to realize that $g$ is just the classical Gauss map $G:M\rightarrow \mathbb{S}^2$ composed by the stereographic projection from the north pole of $\mathbb{S}^2$. Moreover, it is worth noticing that all these formulas still hold if we multiply $\eta$ by an element $e^{i\theta}\in \mathbb{S}^{1}$, thus the new Weierstrass data $(g,e^{i\theta}\eta)$ give another minimal immersion $X_{\theta}:M\rightarrow\mathbb{R}^3$. This leads to the following definition:

\begin{definition}\label{def1.1}
The immersions $X_{\theta}:M\rightarrow\mathbb{R}^3$ with $\theta\in \mathbb{R}/2\pi\mathbb{Z}$ are called associate minimal immersions of $X:M\rightarrow\mathbb{R}^3$.
\end{definition}

We will discuss some geometric data of a minimal immersion with the help of the Weierstrass representation.

Firstly, if we denote locally $\eta$ by $\alpha dz$, then the induced Riemannian metric $g_M$ on $M$ will be determined to be
\begin{align}\label{1.3}
g_M=\frac{1}{4}|\alpha|^2\big(1+|g|^2\big)^2|dz|^2.
\end{align}

The second thing is to determine the Gauss curvature $K$ of $M$. In use of the Riemannian metric $g_M$, the Gauss curvature can be computed as
\begin{align}\label{1.4}
K=-\left[\frac{4|g'|}{|\alpha|(1+|g|^2)^2}\right]^2.
\end{align}

We list in the end some examples of minimal immersions in $\mathbb{R}^3$ with their Weierstrass data:
\begin{example}\label{ex1.1}
(Enneper's Surface) Let $M=\mathbb{C}$, the Weierstrass data are given by $g(z)=z$ and $\eta=dz$.
\end{example}
\begin{example}\label{ex1.2}
(Catenoid) Let $M=\mathbb{C}\backslash\{0\}$, the Weierstrass data are given by $g(z)=z$ and $\eta=\frac{a}{z^2}dz$, with $a\in\mathbb{R}\backslash\{0\}$.
\end{example}

 \section{Ricci metrics and Ricci surfaces}

Andrei Moroianu and Sergiu Moroianu defined in their article \cite{moroianu2015ricci} a class of surfaces called Ricci surfaces which can be regarded as a generalisation of minimal surfaces in $\mathbb{R}^3$. In this section, we are going to introduce some basic properties of Ricci surfaces. Throughout this paper, we only consider orientable Ricci surfaces.

The following theorem is an important fact which has been proven by Gregorio Ricci-Curbastro \cite{ricci1895sulla}:
\begin{theorem}\label{thm 2.1}
A Riemannian surface $(M, g_M)$ with negative Gauss curvature $K<0$ has local isometric immersions as minimal surface in $\mathbb{R}^3$ if and only if one of the three equivalent conditions holds:\\
i). The metric $\sqrt{-K}g_M$ is a flat metric;\\
ii). The metric $(-K)g_M$ is a metric of constant Gauss curvature $1$;\\
iii). The Gauss curvature satisfies 
       \begin{align*}
        K\Delta K + g_M\big(dK,dK\big) + 4K^3=0. 
        \end{align*}
\end{theorem}
The third condition in theorem \ref{thm 2.1} can even be studied without the hypothesis $K<0$. This observation inspires Andrei Moroianu and Sergiu Moroianu to give the following definition.
\begin{definition}\label{def2.1}
A Riemannian surface $(M,g_M)$ whose Gauss curvature $K$ satisfies the following identity
\begin{align}\label{2.1}
K\Delta K+g_M\big(dK,dK\big)+4K^3=0
\end{align}
is called a Ricci surface, the metric $g_M$ is called a Ricci metric.
\end{definition}

From theorem \ref{thm 2.1}, we can see that all the immersed minimal surfaces in $\mathbb{R}^3$ are Ricci surfaces. Moreover, Andrei Moroianu and Sergiu Moroianu proved in their article the following result which shows that a Ricci surface with non-positive Gauss curvature can be locally realized as an immersed minimal surface in $\mathbb{R}^3$.

\begin{theorem}\label{thm2.2}
Let $(M,g_M)$ be a connected Ricci surface, then its Gauss curvature $K$ does not change sign on $M$. If $K\leq0$, then $M$ admits locally a minimal isometric immersion into $\mathbb{R}^3$.
\end{theorem}

These two theorems indicate that the relation between Ricci surfaces with non-positive Gauss curvature and minimal immersions in $\mathbb{R}^3$ is very close. Therefore, we are inspired to use the tools in the theory of minimal immersions in $\mathbb{R}^3$ to study Ricci surfaces. This is possible thanks to the next proposition.
\begin{proposition}\label{prop2.1}
Let $g_{0}=|dz|^2$ be a flat metric on some domain $\Omega\subset\mathbb{C}$ and $f:\Omega\rightarrow\mathbb{R}$ a smooth function. The metric $g_{\Omega}=e^{-2f}g_0$ admits locally a minimal immersion in $\mathbb{R}^3$ if and only if near every $x\in\Omega$ there exists a pair of holomorphic functions $(a,b)$ such that $e^{-f}=|a|^2+|b|^2$.
\end{proposition}
\begin{proof}
(1). $\Rightarrow$ If a such isometric immersion exists, then from the discussion in Section 2, there are Weierstrass data $\eta=\alpha dz$ and $g$ satisfying $e^{-f}=\frac{1}{2}|\alpha|(1+|g|^2)$. Since $\alpha$ can only have a zero of order $2m$ where $g$ has a pole of order $m$, we may find two holomorphic functions $a$ and $b$ such that
\begin{align*}
a^2=\frac{\alpha}{2},\quad b=ag,
\end{align*}
thus $e^{-f}=|a|^2+|b|^2$.

(2). $\Leftarrow$ If there are two holomorphic functions $a$ et $b$ such that $e^{-f}=|a|^2+|b|^2$, then $a$ can be supposed to be non-vanishing on a disc $D$ centered at $x$. Let us define $\alpha:=2a^2$, $g:=\frac{b}{a}$, then $\eta:=\alpha dz$ is a holomorphic 1-form and $g$ is a meromorphic function. Additionally, zeroes of $\eta$ are compatible with the poles of $g$ as mentioned in Section 2.  Taking $(g,\eta)$ as Weierstrass data, we may find an isometric minimal immersion into $\mathbb{R}^3$ whose metric is $g_{\Omega}$.
\end{proof}

Moreover, Calabi has proven the following theorem of rigidity (see \cite{calabi1968quelques}, \cite{benoit2016survey}):

\begin{theorem}\label{thm2.3}
Two minimal isometric immersions from a simply connected surface $M$ into $\mathbb{R}^3$ are associate, up to the action of $Isom(\mathbb{R}^3)$ on $\mathbb{R}^3$.
\end{theorem}

\begin{remark}\label{rem2.1}
This theorem tells us that the minimal isometric immersion from a simply connected surface $M$ into $\mathbb{R}^3$ is unique up to the action of $Isom(\mathbb{R}^3)$ and the $\mathbb{S}^1$-action (associate minimal immersions). Combining proposition \ref{prop2.1} and Calabi's theorem \ref{thm2.3}, we may locally identify the non-positively curved Ricci surface with the equivalent class of its minimal isometric immersions into $\mathbb{R}^3$. This identification will be of great significance in our further discussion.
\end{remark}
We are now interested in studying how the identification mentioned above allows us to do with the Weierstrass data. Firstly, it is easy to be seen from (\ref{1.2}) that the Weierstrass data remain invariant under translations in $\mathbb{R}^3$. Secondly, definition \ref{def1.1} implies that we can multiply $\eta$ by an element $e^{i\theta}\in \mathbb{S}^1$ while keeping $g$ fixed. Thirdly, let us make a rotation in $\mathbb{R}^3$, i.e., applying an element $A\in SO(3,\mathbb{R})$ to Weierstrass representation (\ref{1.2}), then the new surface $\tilde{X}$ and its Gauss map $\tilde{G}$ are written as
\begin{align}\label{2.2}
\tilde{X}=AX,\quad \tilde{G}=AG.
\end{align}
Owing to the fact that $\frac{\partial f}{\partial z}=2\frac{\partial \textrm{Re}f}{\partial z}$ for every holomorphic function $f$, we obtain
\begin{align}\label{2.3}
\tilde{\phi}:=2\frac{\partial \tilde{X}}{\partial z}=A\phi.
\end{align}
It is known that the stereographic projection $\pi$ and rotations preserve angles, so the composite map
\begin{align*}
\pi\circ A\circ\pi^{-1}:\bar{\mathbb{C}}\rightarrow\bar{\mathbb{C}}
\end{align*}
is conformal, hence an element of $Aut(\bar{\mathbb{C}})=PSL(2,\mathbb{C})$. In order to find the new Weierstrass data, we need the following classical result whose proof is standard.
\begin{lemma}\label{lem2.1}
The map
\begin{align*}
\varphi:SO(3,\mathbb{R})&\rightarrow PSL(2,\mathbb{C})\\
A&\mapsto\pi\circ A\circ\pi^{-1}
\end{align*}
is an injective Lie group homomorphism with $\textrm{Im}\varphi=PSU(2)$, hence it induces an isomorphism $SO(3,\mathbb{R})\simeq PSU(2)$.
\end{lemma}



Now we are going to determine the new Weierstrass data after the rotation. As we have already seen, $g=\pi\circ G:M\rightarrow \bar{\mathbb{C}}$ is the Gauss map, hence
\begin{align}\label{2.5}
\tilde{g}=\pi\circ A\circ G=\pi\circ A\circ\pi^{-1}\circ\pi\circ G=\varphi(A).g,
\end{align}
where $\varphi(A).g$ denotes the action of $\varphi(A)$ on $g$. Then by combining (\ref{1.2}), (\ref{2.3}) and (\ref{2.5}), we may get the new Weierstrass data $(\tilde{g},\tilde{\eta})$.




Let us consider the Hopf differential defined by $Q:=dg\otimes\eta$. It is a direct observation from the previous calculation that $Q$ is invariant under rotations in $\mathbb{R}^3$. Therefore, the new Weierstrass data obtained by applying $A\in SO(3,\mathbb{R})$ are computed as
  \begin{equation}\label{2.9}
\left\{
\begin{aligned}
&~\tilde{g}=\varphi(A).(g),\\
&~\tilde{\eta}=\frac{Q}{d\tilde{g}}.\\\end{aligned}
\right.
\end{equation}

From the proof of proposition \ref{prop2.1}, a non-positively curved Ricci surface can be locally equipped with the Weierstrass data. It is worthwhile noticing that the Weierstrass data may not be globally defined on Ricci surfaces in general. However, they are always well-defined on the universal cover of a Ricci surface. In fact, for a non-positively curved Ricci surface $(M,g_M)$, the metric defined by $g_1:=(-K)g_M$ is a metric of constant Gauss curvature 1, possibly with isolated conical singularities at zeroes of $K$ (see \citep{moroianu2015ricci}, \cite{troyanov1989metrics}). It induces on the universal cover $\tilde{M}$ of $M$ a metric $d\sigma^2$ of constant Gauss curvature 1. Then there exists a meromorphic function $g: \tilde{M}\rightarrow \bar{\mathbb{C}}$ such that 
\begin{align}\label{2.10}
d\sigma^2=\frac{4dg d\bar{g}}{(1+\vert g\vert^2)^2}.
\end{align}
Moreover, this function $g$ is unique up to an action of $PSU(2)$ as in (\ref{2.5}). Let us take $g$ as the Gauss map, then the Weierstrass data can be constructed with the help of proposition \ref{prop2.1} and theorem \ref{thm2.3}. To be convenient, we still call it the Weierstrass data of this Ricci surface. A point that should be mentioned is that the metric $d\sigma^2$ is invariant under the action of the deck transformation group $Deck(\tilde{M}/M)$. It is immediate from the uniqueness of $g$ that every deck transformation of $\tilde{M}$ corresponds to an element in $PSU(2)$. More precisely, if we identify the deck transformation group $Deck(\tilde{M}/M)$ with the fundamental group $\pi_1(M)$ of $M$, then there exists a monodromy representation $\rho: \pi_1(M)\rightarrow PSU(2)$ such that
\begin{align}\label{2.11}
g\circ\tau^{-1}=\rho(\tau).g,
\end{align}
for all $\tau\in \pi_1(M)$. This monodromy representation will be very important in our further discussion.

Even though the Hopf differential $Q$ is only defined on the universal cover $\tilde{M}$ in general, its modulus $|Q|$ is actually well-defined on the Ricci surface $M$ itself. To see this, since $Q$ is invariant under the action of $PSU(2)$ and it may differ just by a possible multiplication of $e^{i\theta}\in\mathbb{S}^1$, the modulus $|Q|$ is therefore independant of all these actions. Hence it descends to a well-defined quantity on $M$.

Thanks to the identification mentioned in remark \ref{rem2.1}, all the Weierstrass data under these operations are regarded to be equivalent. This equivalence class of Weierstrass data is exactly what the Weierstrass data of the non-positively curved Ricci surface means. From now on, once a representative pair of Weierstrass data is fixed, we are allowed to apply all the operations discussed above.

To finish this section, we also give an example of Ricci surfaces. For CMC-1 immersions in the hyperbolic 3-space $\mathcal{H}^3$, we may define an analogue of the Weierstrass representation (see \citep{bryant1987surfaces}, \cite{umehara1993complete}). We know from the Lawson's correspondence (see \citep{lawson1970complete}) that CMC-1 immersions in $\mathcal{H}^3$ are locally isometric to minimal immersions in $\mathbb{R}^3$, hence it provides some Ricci surfaces. A typical example is the following: 
\begin{example}\label{ex2.1}
(Catenoid cousins) Let $M=\mathbb{C}\backslash\{0\}$, the representative Weierstrass data are $g(z)=z^{\mu}$ and $\eta=az^{-\mu-1}dz$, with $a, \mu\in\mathbb{R}_{+}^*$ and $\mu\neq 1$. In particular, if $\mu$ is an integer, then it is isometric to a $\mu$-fold cover of catenoid.
\end{example}

\section{Definition of catenoidal ends}

In this section we plan to define catenoidal ends for a non-positively curved Ricci surface with the help of the Weierstrass data. Before beginning, we need some preparation.

\begin{definition}\label{def3.1}
Let $(M,g_M)$ be a complete surface, then the integral
\begin{align}\label{3.1}
K(M):=\int_{M}KdA
\end{align}
is called the total curvature of $M$, where $dA$ is the area element and $K$ is the Gauss curvature.  
\end{definition}

It is well-know that every two-dimensional orientable smooth Riemannian manifold can be equipped with a compatible complex structure so that it may also be regarded as a Riemann surface. We announce here an important theorem proved by Huber \cite{huber1958subharmonic} and Osserman \cite{osserman2013survey}:

\begin{theorem}\label{thm3.1}
Let $(M,g_M)$ be a complete non-positively curved orientable Riemannian surface with finite total curvature. Then there exists a compact Riemann surface $S_k$ of genus $k$ and a finite number of points $p_1,...,p_r$ on $S_k$ such that $M$ is biholomorphic to $S_k\backslash\{p_1,...,p_r\}$.
\end{theorem}

Theorem \ref{thm3.1} permits us to identify a complete non-positively curved Ricci surface which has finite total curvature with a Riemann surface $S_k\backslash\{p_1,...,p_r\}$. It is natural for us to study the Ricci surface with some particular assumptions at the punctured points.

Example \ref{ex1.2} and \ref{ex2.1} showed some complete non-positively curved Ricci surfaces which are biholomorphic to $\bar{\mathbb{C}}\setminus\lbrace0,\infty\rbrace$. Their universal covers are isometric to that of a classical catenoid. In addition, the Hopf differential $Q$ of these examples is meromorphic near each punctured points. We have seen in the previous section that the Hopf differential $Q$ is independent of isometries of $\mathbb{R}^3$, and it is unique up to a multiplication by an element $e^{i\theta}\in\mathbb{S}^1$. Hence we are inspired to study the order of $Q$ at punctured points. At the punctured point $0$, it is not difficult to observe that $\textrm{ord}_{0}\ Q=-2$. Similarly, at the punctured point $\infty$, we replace $z$ by $\frac{1}{w}$, then an easy calculation shows that $\textrm{ord}_{\infty}\ Q=-2$ still holds. The condition $\textrm{ord}\ Q=-2$ is therefore valuable for us. Moreover, it is worth remarking that $\textrm{ord}\ Q=-2$ does not depend on the choice of conformal parameters. For these reasons, we may give the following definition.
\begin{definition}\label{def3.2}
A catenoidal end of a non-positively curved Ricci surface $M\simeq S_k\setminus\lbrace p_1,p_2,...,p_r \rbrace$ with finite total curvature is a punctured point $p_i$ which possesses a neighborhood where the Hopf differential $Q$ is well-defined and meromorphic, satisfying $\textrm{ord}_{p_i}\ Q=-2$.
\end{definition}

For a Ricci surface $M\simeq \bar{\mathbb{C}}\setminus\lbrace p_1,p_2,...,p_r \rbrace$ with $r$ catenoidal ends, the Hopf differential $Q$ is a well-defined holomorphic quadratic differential on $M$. In fact, the fundamental group $\pi_1(M)$ is generated by closed curves $\gamma_1, ..., \gamma_r$ surrounding each punctured point. By definition \ref{def3.2}, $Q$ is independent of all the generators of $\pi_1(M)$, thus well-defined on $M$. However, this property fails to hold for Ricci surfaces with positive genus in general. This is because there will be other generators of the fundamental group $\pi_1(M)$ than the closed curves mentioned above, hence $Q$ may not be well-defined on $M$ itself.   

Taking advantage of (\ref{1.3}) and (\ref{1.4}), we may observe that the modulus $|Q|$ of $Q$ coincides with the metric $\sqrt{-K}g_M$ defined in theorem \ref{thm 2.1}. Therefore, around a catenoidal end which corresponds to $z=0$ of a Ricci surface $(M,g_M)$, the induced flat metric $\sqrt{-K}g_M$ can be locally written as 
\begin{align}\label{3.3}
\sqrt{-K}g_M=\left(\frac{a}{|z|^2}+o\left(\frac{1}{|z|^2}\right)\right)\left|dz\right|^2,
\end{align} 
with $a\in\mathbb{R}\setminus\lbrace 0 \rbrace$. Actually, the consverse is also true. To see this, let us assume that the induced metric $\sqrt{-K}g_M$ of a non-positively curved Ricci surface has a local expression as in (\ref{3.3}) near a punctured point $z=0$. If we consider a closed curve $\gamma$ surrounding this point, then by Calabi's rigidity theorem \ref{thm2.3} and the definition of $Q$, we must have 
\begin{align*}
Q\circ\tau^{-1}=e^{2\pi i\theta}Q
\end{align*}
for some $\theta\in[0,1)$, where $\tau=[\gamma]\in\pi_1(M)$. It can be checked that $P:=z^{-\theta}Q$ is a well-defined meromorphic quadratic differential near $0$. Moreover, the modulus of $P$ takes the form
\begin{align*}
|P|=|z|^{-\theta}|Q|=\left(\frac{a}{|z|^{2+\theta}}+o\left(\frac{1}{|z|^{2+\theta}}\right)\right)\left|dz\right|^2,
\end{align*} 
it implies that $0$ is a pole of $P$ and $\theta=0$. Hence $Q$ is locally well-defined near $0$ and $z=0$ is a catenoidal end of the Ricci surface $M$. Consequently, the catenoidal end of a non-positively curved Ricci surface may be equivalently defined to be a punctured point where the induced flat metric $\sqrt{-K}g_M$ has the expression as in (\ref{3.3}).

Apart from the metric $\sqrt{-K}g_M$, we can say even more about the Weierstrass data of the Ricci surface. The lemma below provides locally an explicit description of the catenoidal end of Ricci surfaces.
\begin{lemma}\label{lem3.1}
For a non-positively curved Ricci surface $(M,g_M)$ with finite total curvature, it has a catenoidal end at the punctured point corresponding to $z=0$ if and only if there exists an element $A\in SO(3,\mathbb{R})$ such that after the rotation, the Weierstrass data can be locally written as

  \begin{equation}\label{3.2}
  \left\{
\begin{aligned}
~\tilde{\eta}&=\tilde{\alpha}dz=\left(\frac{a}{z^{\lambda+1}}+o\left(\frac{1}{z^{\lambda+1}}\right)\right)dz,\\
~\tilde{g}&=bz^{\lambda}+o(z^{\lambda}),\\
\end{aligned}
\right.
\end{equation}
where $a,b\in\mathbb{C}\backslash\{0\}$ and $\lambda\in\mathbb{R}^{*}_{+}$.
\end{lemma}
\begin{proof}
$\Leftarrow$ If there is an element $A\in SO(3,\mathbb{R})$ such that after the rotation, the Weierstrass data $\tilde{\eta}$ and $\tilde{g}$ have form of (\ref{3.2}), then it can be easily checked that $\textrm{ord}\ Q=-2$ is verified at that point. 

$\Rightarrow$ Conversely, given $\textrm{ord}\ Q =-2$ at the punctured point $z=0$, we are obliged to consider an element $\tau=[\gamma]\in\pi_1(M)$, where $\gamma$ is a closed curve around this punctured point. Since each rotation of $\bar{\mathbb{C}}$ is conjugate to a rotation fixing $0$ and $\infty$, in use of lemma \ref{lem2.1}, we may suppose that the function $g$ satisfies
\begin{align*}
g\circ\tau^{-1}=e^{2\pi i\theta}g,
\end{align*}
with $\theta\in[0,1)$. One may verify that $h(z):=z^{-\theta}g$ is a well-defined meromorphic function near $z=0$. The condition that the total curvature of $M$ is finite tells us that $h$ is actually meromorphic at $0$  (see \citep{bryant1987surfaces}, Prop 4), thus the Gauss map $g$ can be locally written as $g=z^{\mu}(\sum\limits_{j=0}^{\infty}{b_jz^j})$ with $\mu\in\mathbb{R}$, $b_j\in\mathbb{C}$ and $b_0\neq0$.  

(1). Assume that $\mu>0$, then it is immediate from the definition of $Q$ that 

\begin{equation*}
\left\{
\begin{aligned}
&~g=z^{\mu}\left(\sum\limits_{j=0}^{\infty}{b_jz^j}\right),\\
&~\eta=z^{-\mu-1}\left(\sum\limits_{j=0}^{\infty}{a_jz^j}\right)dz,\\
\end{aligned}
\right.
\end{equation*}
with $a_0,b_0\in\mathbb{C}\backslash\{0\}$. This is the expression (\ref{3.2}) and we can just take $A=I$. 

(2). Now suppose that $\mu\leq0$.

\quad(i). If $\mu=0$, then $g=b_0+\sum\limits_{j=1}^{\infty}{b_jz^j}$. The condition $\textrm{ord}\ Q =-2$ assures that $g$ is not a constant. Thus we may take a rotation $\tilde{g}=\frac{g-b_0}{\bar{b_0}g+1}$. Since in this case $\tilde{g}$ is meromorphic near $z=0$ and $\tilde{g}(0)=0$, it must have the form 
\begin{equation*}
\tilde{g}=z^n\left(\sum\limits_{j=0}^{\infty}{c_jz^j}\right)
\end{equation*}
for some $n\in\mathbb{N}^*$, where $c_j\in\mathbb{C}$ and $c_0\neq0$.
The isomorphism constructed in lemma \ref{lem2.1} implies immediately the existence of such an element $A\in SO(3,\mathbb{R})$. In addition, the corresponding $\tilde{\eta}$ can be proved to be 
\begin{equation*}
\tilde{\eta}=z^{-n-1}\left(\sum\limits_{j=0}^{\infty}{a_jz^j}\right)dz,
\end{equation*}
with $a_j\in\mathbb{C}$ and $a_0\neq0$. This satisfies the expression (\ref{3.2}).

\quad(ii). In the case $\mu<0$, we will apply the transformation $\tilde{g}=-\frac{1}{g}$. It is easy to see that the new Weierstrass data have the form as in (\ref{3.2}). The corresponding $A\in SO(3,\mathbb{R})$ is also given by lemma \ref{lem2.1}. This completes the proof.
\end{proof}

With the help of this lemma, the Ricci metric can be locally written as
\begin{align}\label{3.4}
g_M=\left(\frac{b}{|z|^{2\lambda+2}}+O\left(\frac{1}{|z|^2}\right)\right)\left|dz\right|^2,
\end{align} 
where $b\in\mathbb{R}\setminus\lbrace 0 \rbrace$ and $\lambda\in\mathbb{R}^{*}_{+}$. Moreover, taking (\ref{3.3}) and (\ref{3.4}) into consideration, we obtain immediately from Ricci's theorem \ref{thm 2.1} a local expression for $(-K)g_M$ which is  
\begin{align}\label{3.5}
(-K)g_M=\left(c|z|^{2\lambda-2}+o\left(|z|^{2\lambda-2}\right)\right)\left|dz\right|^2,
\end{align}
with $c\in\mathbb{R}\setminus\lbrace 0 \rbrace$. Since $\lambda>0$, the metric $(-K)g_M$ has a conical singularity of order $\lambda-1>-1$ at the catenoidal end $z=0$.

It should be noticed that this lemma is just a local result. In order to do a global discussion, we need to study the monodromy representation introduced in the previous section. Inspired by Umehara and Yamada in their article \cite{Masaaki2000metrics}, we will give the following definitions.

\begin{definition}\label{def3.3}
Let $d\sigma^2$ be a metric with constant constant Gauss curvature 1 and finite number of conical singularities on $M$, 

i). We call it an irreducible metric if the image of the corresponding monodromy representation $\rho: \pi_1(M)\rightarrow PSU(2)$ is not abelian;

ii). We call it a non-trivially reducible metric if the image of $\rho$ is abelian but not trivial;

iii). It is called a trivially reducible metric if the image of $\rho$ is trivial. 
\end{definition}

For a non-positively curved Ricci surface $(M,g_M)$ with finite total curvature and more than one ends, if the induced metric $g_1:=(-K)g_M$ is reducible, then we can apply lemma \ref{lem3.1} one by one to all the catenoidal ends. Since the image of the monodromy representation $\rho$ is abelian, the order of this procedure does not influence the final result. Hence in this case, we may do some global analysis. However, this method does not work if the induced metric $g_1$ is irreducible.

\section{$M\simeq\mathbb{C}\setminus\lbrace0\rbrace$ with two catenoidal ends}

We are inspired by Huber's theorem to study non-positively curved Ricci surfaces with catenoidal ends. It is immediate from the Riemann uniformization theorem that the universal cover of a such Ricci surface can only be the whole plane $\mathbb{C}$ or the upper half plane $\mathbb{H}$. In the $\mathbb{C}$ case, the corresponding Ricci surface $M$ is biholomorphic to either $\mathbb{C}\setminus\lbrace 0 \rbrace$ or $\mathbb{C}$. The following theorem affirms that the case when $M\simeq\mathbb{C}$ cannot appear.

\begin{theorem}\label{thm4.2}
There does not exist non-positively curved Ricci surface $M\simeq\mathbb{C}$ with exactly one catenoidal end.
\end{theorem}

\begin{proof}
Let us identify $M$ with $\bar{\mathbb{C}}\setminus\lbrace0\rbrace$, then the catenoidal end corresponds to $0$. Applying lemma \ref{lem3.1}, we may assume that the Gauss map $g$ has a pole at $0$. Since $\pi_1(M)$ is trivial, $g$ can be written as $g=z^{-n}\varphi(z)$, where $n\in\mathbb{N}^{*}$ and $\varphi$ is a meromorphic function on $\bar{\mathbb{C}}$ satisfying $\varphi(0)\neq 0$. The condition for catenoidal ends leads to the fact that $\textrm{ord}\ \eta = -2-(-n-1)=n-1\geq0$ at $0$. Therefore, $\eta$ is a holomorphic 1-form on $\bar{\mathbb{C}}$ thus $\eta\equiv0$. This implies that $Q\equiv0$, which is not possible. 
\end{proof} 

Thanks to theorem \ref{thm4.2}, we may focus on giving a classification of non-positively curved Ricci surfaces $M\simeq\mathbb{C}\setminus\lbrace0\rbrace$ with two catenoidal ends.

\begin{theorem}\label{thm4.1}
A complete non-positively curved Ricci surface $M\simeq\mathbb{C}\backslash\{0\}$ with two catenoidal ends can only be isometric to a catenoid, a catenoid cousin given as in example \ref{ex2.1} or the surface determined by the Weierstrass data
\begin{equation}\label{4.1}
g=z^n+a,\quad\quad \eta=bz^{-n-1}dz,
\end{equation}
where $a,b \in \mathbb{R}^{*}_{+}$ and $n\in\mathbb{N}^*$. Moreover, all these surfaces are mutually non-isometric.
\end{theorem}

\begin{proof}
On the one hand, since $Deck(\tilde{M}/M)\simeq\pi_1(M)\simeq\mathbb{Z}$ is a cyclic group, the proof of lemma \ref{lem3.1} tells us that $h(z):=z^{-\theta}g$ is a meromorphic function on $\bar{\mathbb{C}}$ for some  $\theta\in[0,1)$. Therefore, we may assume that the Gauss map is $g(z)=z^{\mu}f(z)$, where $f$ is a meromorphic function on $\bar{\mathbb{C}}$ and $\mu>0$. Hence it must have the form 
\begin{align}\label{4.2}
g(z)=z^{\mu}\frac{\varphi(z)}{\psi(z)},
\end{align}
where $\varphi$ and $\psi$ are coprime polynomials such that $\psi(0)\neq0$ and $\varphi(0)\neq0$. From lemma \ref{lem3.1} and the fact that $\eta$ can only have a zero of order $2m$ where $g$ has a pole of order $m$, we obtain immediately
\begin{align}\label{4.3}
\eta=\lambda z^{-\mu-1}\psi(z)^2 dz,
\end{align}
where $\lambda\in\mathbb{C}\backslash\{0\}$ is a constant number. On the other hand, as we have explained in the previous section, the Hopf differential $Q=dg\otimes\eta$ is meromorphic on $\bar{\mathbb{C}}$ satisfying $\textrm{ord}\ Q=-2$ at $0$ and $\infty$, thus it can be written as
\begin{align}\label{4.4}
Q=\frac{\nu}{z^2}dz\otimes dz,
\end{align}
where $\nu \in\mathbb{C}\backslash\{0\}$ is a constant number.

Taking (\ref{4.2}), (\ref{4.3}) and (\ref{4.4}) into consideration, we have
\begin{align*}
Q&=dg\otimes \eta\\
&=\lambda z^{-2}[\mu\varphi\psi+z(\varphi'\psi-\varphi\psi')]dz\otimes dz\\
&=\nu z^{-2}dz\otimes dz.
\end{align*}
Comparing these two expressions, we draw the following equation
\begin{align}\label{4.5}
\mu\varphi\psi+z(\varphi'\psi-\varphi\psi')=\frac{\nu}{\lambda}.
\end{align}
Differentiating the two sides of (\ref{4.5}), it becomes
\begin{align*}
\psi[z\varphi''+(1+\mu)\varphi']=\varphi[z\psi''+(1-\mu)\psi'].
\end{align*}
Since $\varphi$ and $\psi$ are coprime, there must be a polynomial $h(z)$ such that
\begin{align}\label{4.6}
z\varphi''+(1+\mu)\varphi'=h(z)\varphi.
\end{align}
The degree of the left-hand side of (\ref{4.6}) is at most $\textrm{deg}\ \varphi-1$, while the degree of the right-hand side should be at least $\textrm{deg}\ \varphi$ unless $h(z)=0$. This comparison tells us that $h(z)=0$. Therefore, (\ref{4.6}) is actually 
\begin{align}\label{4.7}
z\varphi''+(1+\mu)\varphi'=0.
\end{align}
By solving differential equation (\ref{4.7}), we get
\begin{align}\label{4.8}
\varphi=a+bz^{-\mu},
\end{align}
where $a,b\in\mathbb{C}$ are constant numbers.

Similarly, we can get
\begin{align}\label{4.9}
\psi=cz^{\mu}+d,
\end{align}
where $c,d\in\mathbb{C}$ are constant numbers.

Since $\mu>0$ and $\varphi$ is a polynomial, we must have $b=0$.

\quad(i). If $\mu\notin\mathbb{N}$, then $c=0$ because $\psi$ should also be a polynomial. Thus the Weierstrass data of the Ricci surface are
\begin{align*}
g=\frac{a}{d}z^{\mu},\quad\quad \eta=\lambda d^2z^{-\mu-1}dz.
\end{align*}
With a change of variables $u=(\frac{a}{d})^{\frac{1}{\mu}}z$, these expressions can be simplified to be
\begin{align*}
g=u^{\mu},\quad\quad \eta=\lambda adu^{-\mu-1}du.
\end{align*}
Moreover, by applying an $\mathbb{S}^1$-action if necessary, we may suppose $\lambda ad$ to be a real number $\rho>0$, thus the Weierstrass data of the Ricci surface can be written as 
\begin{align}\label{4.10}
g=u^{\mu},\quad\quad \eta=\rho u^{-\mu-1}du
\end{align}
with $\rho\in \mathbb{R}^{*}_{+}$.

\quad(ii). If $\mu\in\mathbb{N}$ and $c=0$, then the Weierstrass data have exactly the same form as in (\ref{4.10}). In these two cases, the Ricci surface $M$ is isometric to a catenoid if $\mu=1$, to a catenoid cousin if $\mu\neq 1$.

\quad(iii). If $\mu\in\mathbb{N}$ and $c\neq 0$, the Weierstrass data will have the form
\begin{align}\label{4.11}
g=\frac{az^{\mu}}{cz^{\mu}+d},\quad\quad \eta=\lambda(cz^{\mu}+d)^2z^{-\mu-1}dz.
\end{align}
Denoting
\begin{align*}
r=\frac{a}{d}\left(\frac{|c|^2}{|a|^2}+1\right), \quad s=\frac{\bar{c}}{\bar{a}},\quad t=\frac{\lambda d^2}{\frac{|c|^2}{|a|^2}+1},
\end{align*}
then a simple computation shows that after a rotation, the Weierstrass data of the Ricci surface become
\begin{align}\label{4.12}
\tilde{g}=rz^{\mu}+s,\quad\quad \tilde{\eta}=tz^{-\mu-1}dz,
\end{align}
where $r,s,t\in\mathbb{C}\backslash\{0\}$. With the help of a change of variables $w=r^{-\frac{1}{\mu}}z$, (\ref{4.12}) is simplified as
\begin{align}\label{4.13}
\tilde{g}=w^{\mu}+s,\quad\quad \tilde{\eta}=trw^{-\mu-1}dw.
\end{align}
Applying an $\mathbb{S}^1$-action if necessary, we may assume $tr$ to be a positive real number $\rho$, hence the Weierstrass data can be written as
\begin{align}\label{4.14}
\tilde{g}=w^{\mu}+s,\quad\quad \tilde{\eta}=\rho w^{-\mu-1}dw,
\end{align}
where $s\in \mathbb{C}\backslash\{0\}$ and $\rho\in \mathbb{R}^{*}_{+}$. Moreover, if we denote $s=\xi e^{i\theta}$ with $\xi\in \mathbb{R}^{*}_{+}$ and $\theta\in\mathbb{R}/2\pi\mathbb{Z}$, then it will be immediate by a rotation and a change of variables that the formula (\ref{4.1}) is obtained. 

From Ricci's theorem \ref{thm 2.1}, we know that a Ricci metric $g_M$ induces natually a metric $(-K)g_M$ which is of constant Gauss curvature 1, possibly with conical singularities. Therefore, we are able to compute the total curvature with the help of Gauss-Bonnet formula with conical singularities
\begin{align}\label{4.22}
\frac{1}{2\pi}\int_{M}(-K)dA=\chi(\bar{M})+\sum_{i=1}^{n}\beta_i,
\end{align}
where $\beta_i>-1$ are the orders of conical singularities (see \citep{troyanov1989metrics}). Taking advantage of the two expressions (\ref{4.10}) and (\ref{4.1}), a direct computation shows that the metric $(-K)g_M$ has two conical singularities of the same order at $0$ and $\infty$, which is $\mu-1$ for (\ref{4.10}) and $n-1$ for (\ref{4.1}). Hence the total curvature is $-4\pi\mu$ for (\ref{4.10}) and $-4\pi n$ for (\ref{4.1}). Since the total curvature is intrinsic, any two of such Ricci surfaces with different total curvatures cannot be isometric.  

Now we are going to prove that two Ricci surfaces with the same total curvature $-4\pi n$ but different pair $(a,b)$ in (\ref{4.1}) could not be isometric. Assuming that there is an isometry $\tau:M_1\rightarrow M_2$, where $M_1, M_2\simeq\mathbb{C}\backslash\{0\}$ are endowed with metrics  defined as in (\ref{4.1}) with $(a_1,b_1)\neq(a_2,b_2)$. Then $\tau$ can be lifted to a biholomorphic map $\tilde{\tau}:\mathbb{C}\rightarrow\mathbb{C}$ through the universal covering map $\textrm{exp}:\mathbb{C}\rightarrow\mathbb{C}\backslash\{0\}$, hence $\tilde{\tau}$ is actually an affine function. More precisely, there exists a pair $(\alpha,\beta)\in \mathbb{C}^2$ such that 
\begin{equation}\label{4.15}
\left\{
\begin{aligned}
&~g=e^{n(\alpha u+\beta)}+a_1=e^{nu}+a_2,\\
&~\eta=b_1\alpha e^{-n(\alpha u+\beta)}du=b_2e^{-nu}du.\\
\end{aligned}
\right.
\end{equation}
By comparing both sides of (\ref{4.15}), we get $\alpha=1$, $a_1=a_2$, $e^{n\beta}=1$ and $b_1=b_2$, contradiction. This completes the proof of the theorem.
\end{proof}

The previous proof is inspired by Umehara and Yamada (see \citep{umehara1993complete}). We give here another proof, which provides a sight of the method that we will use for our further discussion.

\begin{proof}
For the same reason, we suppose that the Weierstrass data have the forms as in (\ref{4.2}) and (\ref{4.3}).

(1). If $\psi$ is holomorphic at $\infty$, then it should be a constant function, thus $\eta=az^{-\mu-1}dz$ with $a\in\mathbb{C}$. This means that $f$ does not have poles on $\mathbb{C}\setminus\lbrace0\rbrace$. In fact, if $f$ has a pole on $\mathbb{C}\setminus\lbrace0\rbrace$, then $\eta$ should have zero at this point, thus $a=0$ and $\eta\equiv0$. This leads to the result $Q\equiv0$, which is a contradiction to the condition of a catenoidal end. Therefore, $f$ is holomorphic on $\mathbb{C}$.

\quad(i). If $f$ is also holomorphic at $\infty$, then $f$ must be a constant function, hence $g=bz^{\mu}$ with $b\in\mathbb{C}$. In this case, we get a catenoid or a catenoid cousin.

\quad(ii). If $f$ has poles at $\infty$, it should be a polynomial 
\begin{align}\label{4.16}
f(z)=\nu(z-a_1)...(z-a_p)
\end{align}
with $\nu, a_i\in\mathbb{C}\setminus\lbrace0\rbrace$ and $p\in\mathbb{N}^*$. Let us make a substitution $z=\frac{1}{w}$, then we get 
\begin{equation*}
\left\{
\begin{aligned}
&~g=\nu w^{-\mu-p}(1-a_1w)...(1-a_pw),\\
&~\eta=-aw^{\mu-1}.\\
\end{aligned}
\right.
\end{equation*}
The condition of catenoidal end at $\infty$ becomes
\begin{align*}
\textrm{ord}\ Q=-p-2=-2,
\end{align*}
which results in $p=0$, contradiction.

(2). If $\psi$ has poles at $\infty$, it must have the form
\begin{align}\label{4.17}
\psi(z)=(z-b_1)...(z-b_q)
\end{align}
with $b_i\in\mathbb{C}\setminus\lbrace0\rbrace$ and $q\in\mathbb{N}^*$. As a consequence, $f$ has the expression
\begin{align}\label{4.18}
f(z)=\frac{\xi(z-c_1)...(z-c_r)}{(z-b_1)...(z-b_q)},
\end{align}
where $r\in\mathbb{N}$ and $\xi,c_j\in\mathbb{C}\setminus\lbrace0\rbrace$ satisfying $b_i\neq c_j$ for all $(i,j)$. Replacing $z$ by $\frac{1}{w}$, we obtain
\begin{equation*}
\left\{
\begin{aligned}
&~g=\xi w^{-\mu-r+q}\frac{(1-c_1w)...(1-c_rw)}{(1-b_1w)...(1-b_qw)},\\
&~\eta=-\lambda w^{\mu-2q-1}(1-b_1w)^2...(1-b_qw)^2.\\
\end{aligned}
\right.
\end{equation*}

\quad(i). If $-\mu-r+q\neq0$, then the condition of catenoidal end at $\infty$ implies $-r-q-2=-2$, thus $r+q=0$, which is not possible.

\quad(ii). The only case left is $\mu=q-r\in\mathbb{N}^*$. Applying $\tilde{g}=\frac{g-\xi}{\bar{\xi}g+1}$, we get the following formula
\begin{align}\label{4.19}
\tilde{g}=\xi\frac{(1-c_1w)...(1-c_rw)-(1-b_1w)...(1-b_qw)}{\xi\bar{\xi}(1-c_1w)...(1-c_rw)+(1-b_1w)...(1-b_qw)}.
\end{align}
On the one side, it can be seen that $\textrm{ord}\ \tilde{g}\leq q$ because $q>r$. On the other side, since $\textrm{ord}\ \tilde{\eta}=\textrm{ord}\ \eta=\mu-2q-1=-q-r-1$, we can compute
\begin{align*}
\textrm{ord}\ d\tilde{g}=-2-\textrm{ord}\ \tilde{\eta}=q+r-1\geq0, 
\end{align*}
hence $\textrm{ord}\ \tilde{g}=q+r$. We deduce from this comparison that $r=0$. Combining (\ref{4.19}) with the fact that $\textrm{ord}\ \tilde{g}=q$, we may observe that 
  \begin{equation*}
\left\{
\begin{aligned}
&~I_1:=\sum_{1\leq i\leq n}c_i=0,\\
&~I_2:=\sum_{1\leq i<j\leq n}c_ic_j=0,\\
&~...\\
&~I_{n-1}:=\sum_{1\leq i_1<i_2<...<i_{n-1}\leq n}c_{i_1}...c_{i_{n-1}}=0,\\
&~I_n:=c_1c_2...c_n\neq0.
\end{aligned}
\right.
\end{equation*}
Consequently, the Weierstrass data take the same form as in (\ref{4.11}), hence the conclusion can be drawn by utilizing the same argument as in the first proof. 
\end{proof}

From the proofs of this theorem, we may easily obtain the following conclusion.
\begin{corollary}\label{cor4.1}
A complete non-positively curved Ricci surface $M\simeq\mathbb{C}\backslash\{0\}$ whose total curvature is $-4\pi\mu$ with $\mu\in\mathbb{R}^{*}_{+}\setminus\mathbb{N}$ must be isometric to a catenoid cousin.
\end{corollary}
 
\begin{remark}\label{rem4.1}
Umehara and Yamada have proven in \cite{umehara1993complete} a similar result for complete CMC-1 surfaces in the hyperbolic 3-space $\mathcal{H}^3$. Then by the Lawson correspondence, these surfaces correspond to Ricci surfaces. It is worthwhile to mention that theorem \ref{thm4.1} gives actually more surfaces than theirs. In fact, the main difference between our theorem and the result of Umehara and Yamada is the catenoid case and that they have for the case (\ref{4.1}) a supplementary condition 
\begin{align}\label{4.20}
n^2+4bn=m^2
\end{align}
with $m\in \mathbb{N}^{*}$. Condition (\ref{4.20}) implies that $b=\frac{m^2-n^2}{4n}\in \mathbb{Q}$, but we just require $b$ to be a real number. Then the previous argument says that we have indeed more surfaces. Moreover, up to applying a homothety, our result implies that every complete non-positively curved Ricci surface $M\simeq\mathbb{C}\backslash\{0\}$ with two catenoidal ends admits an isometric CMC-1 immersion in $\mathcal{H}^3$.
\end{remark}

\begin{remark}\label{rem4.2}
Troyanov has classified all the metrics of constant Gauss curvature 1 on $\bar{\mathbb{C}}$ with two conical singularities at $0$ and $\infty$ (see \citep{troyanov1989metrics}). In fact, he affirms that if $d\sigma^2$ is a such metric, then there exists $\lambda\geq0$ and $\beta>-1$, such that $d\sigma^2$ has the expression
\begin{align}\label{4.21}
d\sigma^2=(2+2\beta )^2\frac{|z|^{2\beta}|dz|^2}{(|1+\lambda z^{\beta+1}|^2+|z|^{2\beta+2})^2},
\end{align}
where $\beta$ is either an integer, or $\lambda=0$. In theorem \ref{thm4.1}, the condition for catenoidal ends at $0$ and $\infty$ implies that we should have a flat metric $|Q|=\nu\frac{|dz|^2}{|z|^2}$ with $\nu\in\mathbb{R}^{*}_{+}$. Let us fix this flat metric, then each metric $d\sigma^2$ in Troyanov's theorem gives rise to a desired Ricci surface for us (see \citep{moroianu2015ricci}, Lemma 2.3). Conversely, for every Ricci surface in theorem \ref{thm4.1}, since the induced flat metric $\sqrt{-K}g_M$ coincides with $|Q|$ which does not vanish on $\mathbb{C}\setminus\lbrace0\rbrace$, its Gauss curvature $K$ has no zeroes on $\mathbb{C}\setminus\lbrace0\rbrace$. Hence the metric $(-K)g_M$ is of constant Gauss curvature 1 with two conical singularities at $0$ and $\infty$. It is a direct verification that $(-K)g_M$ has the same expression as in (\ref{4.21}). Therefore, our result and Troyanov's theorem are equivalent.
\end{remark}

\section{$M\simeq\mathbb{C}\setminus\lbrace0,1\rbrace$ with three catenoidal ends}

After giving a complete classification for non-positively curved Ricci surface $M\simeq\mathbb{C}\setminus\lbrace0\rbrace$ with two catenoidal ends, we are naturally interested in studying the case when the Ricci surface is topologically $M\simeq\mathbb{C}\setminus\lbrace0,1\rbrace$ with three catenoidal ends at $0,1,\infty$. However, since $\pi_1(M)\simeq\mathbb{Z}*\mathbb{Z}$ is not abelian, we have to pay attention to the monodromy representation $\rho: \pi_1(M)\rightarrow PSU(2)$, as mentionned in the end of Section 4.

We will only consider the case when the induced metric $(-K)g_M$ is reducible. Let us denote $\tau_1,\tau_2$ the two generators of $\pi_1(M)\simeq\mathbb{Z}*\mathbb{Z}$ which are represented by two closed curves $\gamma_1, \gamma_2$ surrounding $0$ and $1$, respectively. On the one hand, since the image of $\rho$ is abelian, we may choose a proper $g$ such that 
\begin{align*}
g\circ{\tau_j}^{-1}=e^{2\pi i\theta_j}g,\quad j=1,2,
\end{align*}
where $\theta_j\in[0,1)$. Imitating the proof of lemma \ref{lem3.1}, one can verify that the function $h(z):=(z-1)^{-\theta_2}z^{-\theta_1}g$ is meromorphic on $\bar{\mathbb{C}}$, thus a rational function. Therefore, $g$ should be written as 
\begin{align*}
g=(z-1)^{\theta_2+p}z^{\theta_1+q}\frac{\varphi(z)}{\psi(z)},
\end{align*}
where $p,q\in\mathbb{Z}$ and $\varphi,\psi$ are coprime polynomials without zeroes at $0,1$. By interchanging the roles of $0,1$ and applying a rotation if necessary, we may devide the Weierstrass data into two possible situations:\\
Case 1:
  \begin{equation}\label{6.1}
\left\{
\begin{aligned}
&~g=\lambda (z-1)^{\mu} z^\nu \frac{\varphi(z)}{\psi(z)},\\
&~\eta=\kappa (z-1)^{-\mu -1}z^{-\nu -1} \psi^2(z)dz,\\
\end{aligned}
\right.
\end{equation}
where $\lambda, \kappa\in\mathbb{C}\setminus\lbrace0\rbrace$, $\mu>0$, $\nu \in\mathbb{R}^*$ and $\varphi,\psi$ are coprime polynomials without zeroes at $0,1$. \\
Case 2:
 \begin{equation}\label{6.2}
\left\{
\begin{aligned}
&~g=\lambda (z-1)^{\mu} \frac{\varphi(z)}{\psi(z)},\\
&~\eta=\kappa (z-1)^{-\mu -1}z^{-\nu -1} \psi^2(z)dz,\\
\end{aligned}
\right.
\end{equation}
where $\lambda, \kappa\in\mathbb{C}\setminus\lbrace0\rbrace$, $\mu>0$, $\nu \in\mathbb{N}^*$ and $\varphi,\psi$ are coprime polynomials without zeroes at $0,1$.

In fact, if $\theta_2+p\neq0$, then up to replacing $g$ by $\frac{1}{g}$, we will obtain (\ref{6.1}) or (\ref{6.2}). In the case when $\theta_2+p=0$ but $\theta_1+q\neq0$, we may apply a Mobius transformation $\sigma(z)=1-z$ to interchange the role of $0$ and $1$, then we are back to the first case. For the last case, if both $\theta_2+p$ and $\theta_1+q$ are zero, then we may use a rotation to get a new $g$ which has zeroes at $1$. Since such a rotation does not change the orders of $dg$ and $\eta$ at $1$, we will find (\ref{6.1}) or (\ref{6.2}) again. 

On the other hand, we know from Section 4 that the Hopf differential $Q$ is holomorphic on $M$ with three poles of order $2$ at $0,1$ and $\infty$, thus it can be determined to be 
\begin{align}\label{6.3}
Q=\alpha\frac{(z-a_1)(z-a_2)}{z^2(z-1)^2}dz\otimes dz,
\end{align}
where $a_1,a_2\in\mathbb{C}\setminus\lbrace0,1\rbrace$ and $\alpha\in\mathbb{C}\setminus\lbrace0\rbrace$. Combining (\ref{6.1}), (\ref{6.2}) with the definition of Hopf differential $Q=dg\otimes\eta$, we can see that no matter which case shall we take, $dg$ actually has a uniform expression 
\begin{align}\label{6.4}
dg=\frac{\alpha}{\kappa}(z-1)^{\mu -1}z^{\nu -1}\frac{(z-a_1)(z-a_2)}{\psi^2(z)}dz.
\end{align}
Therefore, we obtain two differential equations with respect to the two cases:\\
For case 1, we have 
\begin{align}\label{6.5}
[(\mu+\nu)z-\nu]\varphi\psi + z(z-1)(\varphi'\psi-\varphi\psi')=\frac{\alpha}{\lambda\kappa}(z-a_1)(z-a_2).
\end{align}
For case 2, the equation is
\begin{align}\label{6.6}
\mu\varphi\psi + (z-1)(\varphi'\psi-\varphi\psi')=\frac{\alpha}{\lambda\kappa}z^{\nu-1}(z-a_1)(z-a_2).
\end{align}
Any solution $\varphi,\psi$ of (\ref{6.5}) or (\ref{6.6}) will give rise to a desired Ricci surface. However, these two equations are very complicated to solve directly, so we need to do a more general discussion.
   
(1). Suppose that $a_1\neq a_2$. If $\psi$ has a zero $c$ of order higher than $1$, then by analysing (\ref{6.5}) and (\ref{6.6}), we have $(z-c)|(z-a_1)(z-a_2)$, hence $c=a_1$ or $c=a_2$ and the order of $(z-c)$ is exactly $2$.

\quad (i). If $\psi(z)=(z-a_1)^2(z-a_2)^2\widehat{\psi}$, then $\widehat{\psi}$ should be a polynomial whose zeroes are all distinct. Let us denote $\widehat{\psi}(z)=\prod_{j=1}^{n}(z-b_j)$. The existence of $g$ requires that all the residues at poles of $dg$ (zeroes of $\psi$) to be zero, whence we get the following conditions:
\begin{align}\label{6.7}
\frac{\mu-1}{b_i-1}+\frac{\nu-1}{b_i}-\frac{3}{b_i-a_1}-\frac{3}{b_i-a_2}-\sum_{j\neq i}\frac{2}{b_i-b_j}=0,\quad \forall i\in\lbrace 1,2,...,n\rbrace,
\end{align}
\begin{align}\label{6.8}
\notag &\frac{(\mu-1)(\mu-2)}{(a_1-1)^2}+\frac{2(\mu-1)(\nu-1)}{a_1(a_1-1)}+\frac{(\nu-1)(\nu-2)}{a_1^2}+\frac{12}{(a_1-a_2)^2}\\
-&\left(\frac{\mu-1}{a_1-1}+\frac{\nu-1}{a_1}\right)\left(\frac{6}{a_1-a_2}+\sum_{j=1}^{n}\frac{4}{a_1-b_j}\right)+\frac{6}{a_1-a_2}\left(\sum_{j=1}^{n}\frac{2}{a_1-b_j}\right)\\
\notag +&\left(\sum_{j=1}^{n}\frac{2}{a_1-b_j}\right)^2+\sum_{j=1}^{n}\frac{2}{(a_1-b_j)^2}=0,
\end{align}
and
\begin{align}\label{6.9}
\notag &\frac{(\mu-1)(\mu-2)}{(a_2-1)^2}+\frac{2(\mu-1)(\nu-1)}{a_2(a_2-1)}+\frac{(\nu-1)(\nu-2)}{a_2^2}+\frac{12}{(a_2-a_1)^2}\\
 -&\left(\frac{\mu-1}{a_2-1}+\frac{\nu-1}{a_2}\right)\left(\frac{6}{a_2-a_1}+\sum_{j=1}^{n}\frac{4}{a_2-b_j}\right)+\frac{6}{a_2-a_1}\left(\sum_{j=1}^{n}\frac{2}{a_2-b_j}\right)\\
\notag +&\left(\sum_{j=1}^{n}\frac{2}{a_2-b_j}\right)^2+\sum_{j=1}^{n}\frac{2}{(a_2-b_j)^2}=0.
\end{align}

\quad (ii). Let us consider the case when $\psi(z)=(z-a_1)^2\widehat{\psi}$, where $\widehat{\psi}$ is the same as in case (i). A similar computation yields that
\begin{align}\label{6.10}
\frac{\mu-1}{b_i-1}+\frac{\nu-1}{b_i}-\frac{3}{b_i-a_1}+\frac{1}{b_i-a_2}-\sum_{j\neq i}\frac{2}{b_i-b_j}=0,\quad \forall i\in\lbrace 1,2,...,n\rbrace,
\end{align}
and
\begin{align}\label{6.11}
\notag &\frac{(\mu-1)(\mu-2)}{(a_1-1)^2}+\frac{2(\mu-1)(\nu-1)}{a_1(a_1-1)}+\frac{(\nu-1)(\nu-2)}{a_1^2}+\frac{2(\nu-1)}{a_1(a_1-a_2)}\\
+&\frac{2(\mu-1)}{(a_1-1)(a_1-a_2)}
-\left(\sum_{j=1}^{n}\frac{4}{a_1-b_j}\right)\left(\frac{\mu-1}{a_1-1}+\frac{\nu-1}{a_1}+\frac{1}{a_1-a_2}\right)\\
\notag +&\left(\sum_{j=1}^{n}\frac{2}{a_1-b_j}\right)^2+\sum_{j=1}^{n}\frac{2}{(a_1-b_j)^2}=0,
\end{align}
should be satisfied.

\quad (iii). The last case is when all the zeroes of $\psi$ are distinct. We will keep the notation as in the two previous cases, then we must have 
\begin{align}\label{6.12}
\frac{\mu-1}{b_i-1}+\frac{\nu-1}{b_i}+\frac{1}{b_i-a_1}+\frac{1}{b_i-a_2}-\sum_{j\neq i}\frac{2}{b_i-b_j}=0,\quad \forall i\in\lbrace 1,2,...,n\rbrace.
\end{align}

(2). Now let us assume that $a_1=a_2$. In this case, if $\psi$ has a zero with multiplicity, then it could be $(z-a_1)^2$ or $(z-a_1)^3$.

\quad (i). For the case when $\psi$ does not have zeroes with multiplicity, we easily get
\begin{align}\label{6.13}
\frac{\mu-1}{b_i-1}+\frac{\nu-1}{b_i}+\frac{2}{b_i-a_1}-\sum_{j\neq i}\frac{2}{b_i-b_j}=0,\quad \forall i\in\lbrace 1,2,...,n\rbrace.
\end{align} 

\quad (ii). If $\psi(z)=(z-a_1)^2\widehat{\psi}$, then the conditions are
\begin{align}\label{6.14}
\frac{\mu-1}{b_i-1}+\frac{\nu-1}{b_i}-\frac{2}{b_i-a_1}-\sum_{j\neq i}\frac{2}{b_i-b_j}=0,\quad \forall i\in\lbrace 1,2,...,n\rbrace,
\end{align}
and
\begin{align}\label{6.15}
\frac{\mu-1}{a_1-1}+\frac{\nu-1}{a_1}-\sum_{j=1}^{n}\frac{2}{a_1-b_j}=0,\quad \forall i\in\lbrace 1,2,...,n\rbrace.
\end{align}

\quad (iii). The only case left is when $\psi(z)=(z-a_1)^3\widehat{\psi}$. Similarly, the conditions can be computed to be 
\begin{align}\label{6.16}
\frac{\mu-1}{b_i-1}+\frac{\nu-1}{b_i}-\frac{4}{b_i-a_1}-\sum_{j\neq i}\frac{2}{b_i-b_j}=0,\quad \forall i\in\lbrace 1,2,...,n\rbrace,
\end{align} 
and 
\begin{align}\label{6.17}
\notag &2\left(\frac{\mu-1}{a_1-1}+\frac{\nu-1}{a_1}\right)\left[11(\sum_{j=1}^{n}\frac{1}{a_1-b_j})^2+5\sum_{i\neq j}\frac{1}{(a_1-b_i)(a_1-b_j)}\right]\\
\notag +&3\left[\frac{(\mu-1)(\mu-2)}{(a_1-1)^2}+\frac{2(\mu-1)(\nu-1)}{a_1(a_1-1)}+\frac{(\nu-1)(\nu-2)}{a_1^2}\right]\left(\sum_{j=1}^{n}\frac{2}{a_1-b_j}\right)\\
+&\frac{(\mu-1)(\mu-2)(\mu-3)}{(a_1-1)^3}+\frac{3(\mu-1)(\nu-1)(\mu-2)}{a_1(a_1-1)^2}\\
\notag +&\frac{3(\mu-1)(\nu-1)(\nu-2)}{a_1^2(a_1-1)}+\frac{(\nu-1)(\nu-2)(\nu-3)}{a_1^3}\\
\notag -&\sum_{j=1}^{n}\frac{32}{(a_1-b_j)^3}-\sum_{i\neq j}\frac{60}{(a_1-b_i)^2(a_1-b_j)}\\
\notag -&\sum_{i\neq j\neq k}\frac{16}{(a_1-b_i)(a_1-b_j)(a_1-b_k)}=0.
\end{align}
If a Ricci surface induces a reducible metric $(-K)g_M$ of constant Gauss curvature 1, then it must satisfy equation (\ref{6.5}) or (\ref{6.6}) and one of the previous cases; the converse is also true. Therefore, so as to find all the desired Ricci surfaces, we may firstly solve algebraic equations (\ref{6.7})--(\ref{6.17}) of each case to completely determine the polynomial $\psi$. Then the problem is reduced to searching solutions to first order equations (\ref{6.5}) and (\ref{6.6}) with a known polynomial $\psi$. We summarize these arguments in the following theorem:

\begin{theorem}\label{thm6.1}
Given $\mu>0$, $\nu\in\mathbb{R}^*$ and $a_1,a_2\in\mathbb{C}\setminus\lbrace0,1\rbrace$, a complete non-positively curved Ricci surface $M\simeq\mathbb{C}\setminus\lbrace0,1\rbrace$ with a reducible metric $(-K)g_M$ and three catenoidal ends exists if only if one of the following cases is verified:

(i). If $a_1\neq a_2$, there exist $n\in\mathbb{N}$ and mutually distinct $b_1,...,b_n\in\mathbb{C}\setminus\lbrace0,1,a_1,a_2\rbrace$ such that (\ref{6.7})--(\ref{6.9}) are satisfied and equation (\ref{6.5}) or (\ref{6.6}) admits a polynomial solution $\varphi$ with $\psi=(z-a_1)^2(z-a_2)^2\prod_{j=1}^{n}(z-b_j)$;

(ii). If $a_1\neq a_2$, there exist $n\in\mathbb{N}$ and mutually distinct $b_1,...,b_n\in\mathbb{C}\setminus\lbrace0,1,a_1,a_2\rbrace$ such that (\ref{6.10})--(\ref{6.11}) are satisfied and equation (\ref{6.5}) or (\ref{6.6}) admits a polynomial solution $\varphi$ with $\psi=(z-a_1)^2\prod_{j=1}^{n}(z-b_j)$;

(iii). If $a_1\neq a_2$, there exist $n\in\mathbb{N}$ and mutually distinct $b_1,...,b_n\in\mathbb{C}\setminus\lbrace0,1,a_1,a_2\rbrace$ such that (\ref{6.12}) are satisfied and equation (\ref{6.5}) or (\ref{6.6}) admits a polynomial solution $\varphi$ with $\psi=\prod_{j=1}^{n}(z-b_j)$;

(iv). If $a_1=a_2$, there exist $n\in\mathbb{N}$ and mutually distinct $b_1,...,b_n\in\mathbb{C}\setminus\lbrace0,1,a_1\rbrace$ such that (\ref{6.13}) are satisfied and equation (\ref{6.5}) or (\ref{6.6}) admits a polynomial solution $\varphi$ with $\psi=\prod_{j=1}^{n}(z-b_j)$;

(v). If $a_1=a_2$, there exist $n\in\mathbb{N}$ and mutually distinct $b_1,...,b_n\in\mathbb{C}\setminus\lbrace0,1,a_1\rbrace$ such that (\ref{6.14})--(\ref{6.15}) are satisfied and equation (\ref{6.5}) or (\ref{6.6}) admits a polynomial solution $\varphi$ with $\psi=(z-a_1)^2\prod_{j=1}^{n}(z-b_j)$;

(vi). If $a_1=a_2$, there exist $n\in\mathbb{N}$ and mutually distinct $b_1,...,b_n\in\mathbb{C}\setminus\lbrace0,1,a_1\rbrace$ such that (\ref{6.16})--(\ref{6.17}) are satisfied and equation (\ref{6.5}) or (\ref{6.6}) admits a polynomial solution $\varphi$ with $\psi=(z-a_1)^3\prod_{j=1}^{n}(z-b_j)$.
\end{theorem}

\begin{remark}\label{rem6.1}
In the case when $M\simeq\mathbb{C}\setminus\lbrace0,1\rbrace$, we are not allowed to use the method mentioned in remark \ref{rem4.2} to study Ricci surfaces. This is because the flat metric $|Q|$ vanishes at $a_1, a_2\in\mathbb{C}\setminus\lbrace0,1\rbrace$, which should coincide with zeroes of the Gauss curvature $K$. It means that the hypothesis $V>0$ of Lemma 2.3 in \citep{moroianu2015ricci} fails to be satisfied. Therefore, we are obliged to take advantage of the method as above.
\end{remark}

Now we plan to use theorem \ref{thm6.1} to do some classification for Ricci surfaces with a reducible metric $(-K)g_M$ and with finite total curvature $-4\pi$. Denote $\textrm{deg}\ \psi=N$ and $\textrm{deg}\ \varphi=m$, then by using Gauss-Bonnet formula with conical singularities (\ref{4.22}) again, we are able to compute the total curvature. The expression of total curvature should be divided into several cases:

(1). Suppose that the Weierstrass data take the form as in (\ref{6.1}) with $\nu>0$. The induced metric $(-K)g_M$ can be determined to be 
\begin{align*}
(-K)g_M=\frac{4|\lambda\alpha|^2}{|\kappa|^2}\cdot\frac{|z-1|^{2(\mu-1)}|z|^{2(\nu-1)}|z-a_1|^2|z-a_2|^2}{(|\psi|^2+|\lambda|^2|z-1|^{2\mu}|z|^{2\nu}|\varphi|^2)^2}|dz|^2.
\end{align*} 

\quad(a). If $N>\mu+\nu+m$, then this metric has conical sinularities of order $\nu-1,\mu-1,2N-2-\mu-\nu,1,1$ at $0,1,\infty,a_1,a_2$, respectively. Hence the total curvature is $-4\pi N$. The hypothesis that the total curvature should be $-4\pi$ implies $N=1$, thus $m=0$ and $0<\mu+\nu<1$. In this case, we have $\varphi=1$ and $\psi(z)=z-b$ for some $b\in\mathbb{C}\setminus\lbrace0,1\rbrace$. Comparing two sides of equation (\ref{6.5}), we can see that if
\begin{align}\label{6.18}
\nu(\nu-1)=(\mu+\nu-1)[\nu(a_1+a_2)-(\mu+\nu)a_1a_2]
\end{align}
is verified, then $b=\frac{(\mu+\nu-1)a_1a_2}{\nu}$ is uniquely determined. Moreover, theorem \ref{thm6.1} tells us that $b$ should satisfy condition (\ref{6.12}) or (\ref{6.13}).  

\quad(b). If $N=\mu+\nu+m$ and $N+m\geq2$, then the total curvature is also $-4\pi N$. Hence we must have $N=1$ and $m=0$, which is not possible.

\quad(c). If $N<\mu+\nu+m$, then the total curvature is $-4\pi(\mu+\nu+m)$. The condition $\mu+\nu+m=1$ tells us that $m=0$ and $N=0$. In this case, $\psi=\varphi=1$ and that condition (\ref{6.12}) or (\ref{6.13}) is obviously satisfied. However, equation (\ref{6.5}) will not hold because the degrees on the two sides are not coincident. Hence this case cannot appear. 

(2). Now assume that the Weierstrass data also take the form as in (\ref{6.1}) but with $\nu<0$. In this case, we have 
\begin{align*}
(-K)g_M=\frac{4|\lambda\alpha|^2}{|\kappa|^2}\cdot\frac{|z-1|^{2(\mu-1)}|z|^{-2(\nu+1)}|z-a_1|^2|z-a_2|^2}{(|z|^{-2\nu}|\psi|^2+|\lambda|^2|z-1|^{2\mu}|\varphi|^2)^2}|dz|^2.
\end{align*} 

\quad(d). If $N>\mu+\nu+m$, then the total curvature is $-4\pi(N-\nu)$, thus we get $N=0$ and $\nu=-1$. Moreover, $1>\mu+m$ implies $m=0$. The same argument as in (c) shows that this is impossible.

\quad(e). If $N<\mu+\nu+m$, then the total curvature is $-4\pi(\mu+m)$. Again, we will obtain $m=0$ and $N=0$, not possible.

\quad(f). If $N=\mu+\nu+m$ and $N+m\geq2$, then the total curvature is also $-4\pi(N-v)$. This leads to $N=m=0$, impossible.

(3). Suppose now that the Weierstrass data have the expression as in (\ref{6.2}), where $\nu\in\mathbb{N}^*$. A similar computation shows 
\begin{align*}
(-K)g_M=\frac{4|\lambda\alpha|^2}{|\kappa|^2}\cdot\frac{|z-1|^{2(\mu-1)}|z|^{2(\nu-1)}|z-a_1|^2|z-a_2|^2}{(|\psi|^2+|\lambda|^2|z-1|^{2\mu}|\varphi|^2)^2}|dz|^2.
\end{align*} 

\quad(g). If $N>\mu+m$, the total curvature is $-4\pi N$, thus $N=1$ and $m=0$. In this case, $\varphi=1$ and $\psi(z)=z-b$ for some $b$ satisfying (\ref{6.12}) or (\ref{6.13}). However, the degree of the left-hand side of (\ref{6.6}) is $1$ while the degree of the right-hand side is at least $2$, contradiction.

\quad(h). If $N<\mu+m$, the total curvature is $-4\pi(\mu+m)$. This implies $m=N=0$. A similar argument as in (c) shows that equation (\ref{6.6}) will never hold, thus it is impossible. 

\quad(i). If $N=\mu+m$ and $N+m\geq\nu+2$, the total curvature is $-4\pi N$. In this case, we get $N=1$ and $m=0$, but $N+m=1$ contradicts to the condition $N+m\geq\nu+2$.

These arguments can be concluded as the following theorem:

\begin{theorem}\label{thm6.2}
Assuming that $M\simeq\mathbb{C}\setminus\lbrace0,1\rbrace$ is a complete Ricci surface with three catenoidal ends and a reducible induced metric $(-K)g_M$, if it has finite total curvature $-4\pi$, then it must be isometric to the surface determined by the Weierstrass data 
  \begin{equation}\label{6.19}
\left\{
\begin{aligned}
&~g=\lambda(z-1)^{\mu} z^\nu \frac{1}{z-\frac{(\mu+\nu-1)a_1a_2}{\nu}},\\
&~\eta=\kappa (z-1)^{-\mu -1}z^{-\nu -1} \left(z-\frac{(\mu+\nu-1)a_1a_2}{\nu}\right)^2dz,\\
\end{aligned}
\right.
\end{equation}
where $a_1,a_2\in\mathbb{C}\setminus\lbrace0,1\rbrace$, $\lambda,\kappa\in\mathbb{C}\setminus\lbrace0\rbrace$, $\mu,\nu\in\mathbb{R}^*_+$ satisfying $0<\mu+\nu<1$, (\ref{6.18}) and (\ref{6.12}) (or (\ref{6.13})).
\end{theorem}

Moreover, we get instantly the following result from the previous discussion.

\begin{theorem}\label{thm6.3}
There is no complete Ricci surface $M\simeq\mathbb{C}\setminus\lbrace0,1\rbrace$ with three catenoidal ends which induces a reducible metric $(-K)g_M$ and has finite total curvature $-4\pi l$ for $0<l<1$.
\end{theorem}

With the help of theorem \ref{thm6.2}, we may construct an example.

\begin{example}\label{ex6.1}
Taking $\mu=\frac{1}{4}$, $\nu=\frac{1}{2}$, then condition (\ref{6.18}) becomes 
\begin{align*}
4=2(a_1+a_2)-3a_1a_2.
\end{align*}
By a direct computation, we can verify that (\ref{6.12}) or (\ref{6.13}) is true for $b=-\frac{1}{2}a_1a_2$. Therefore, we may take $a_1=\frac{6}{5}$ and $a_2=-1$, then (\ref{6.19}) will give us a desired example of Ricci surface. 
\end{example}

For Ricci surfaces $M\simeq\mathbb{C}\setminus\lbrace0,1\rbrace$ with total curvature $-4\pi l$ for $l>1$, things will become much more complicated. However, it is always beneficial to mention the trinoid constructed by Jorge and Meeks in \citep{jorge1983topology}, which is a typical example of such Ricci surfaces whose induced metric $(-K)g_M$ is reducible.

\begin{example}\label{ex6.2}
(Trinoid) The classical Weierstrass data of a trinoid are
\begin{align*}
g=z^2,\quad \eta=\frac{dz}{(z^3-1)^2},
\end{align*}
defined on $\bar{\mathbb{C}}\setminus\lbrace 1,\xi,\xi^2 \rbrace$, where $\xi=e^{\frac{2\pi i}{3}}$. After a Mobius transformation and a rotation, we may find a new pair of Weierstrass data defined on $\mathbb{C}\setminus\lbrace0,1\rbrace$ as 
\begin{equation}\label{6.20}
\left\{
\begin{aligned}
&~g=\frac{\xi-1}{\xi+1}\cdot\frac{(w-1)(w+1)}{(w-2-\sqrt{3})(w-2+\sqrt{3})},\\
&~\eta=\frac{-1}{9(2\xi+1)}\cdot\frac{(w-2-\sqrt{3})^2(w-2+\sqrt{3})^2}{w^2(w-1)^2}dw.\\
\end{aligned}
\right.
\end{equation}
Then a little computation shows $a_1=\frac{1+\sqrt{3}i}{2}$ and $a_2=\frac{1-\sqrt{3}i}{2}$. It can be checked that conditions (\ref{6.12}) and equation (\ref{6.6}) are satisfied. Additionally, by using Gauss-Bonnet formula (\ref{4.22}) again, we know that its total curvature is $-8\pi$. 
\end{example}

\section{Higher genus Ricci surfaces}

As we have seen in Section 6, it is very difficult to give a total classification of Ricci surfaces with catenoidal ends in general. However, the existence of Ricci surfaces $M\simeq \bar{\mathbb{C}}\setminus\lbrace p_1,p_2,...,p_n \rbrace$ with $n$ catenoidal ends ($n\geq2$) is well-known, which is given by the n-noid (see \cite{weber2005classical}).
\begin{example}\label{ex7.1}
(n-noid) Suppose $M\simeq \bar{\mathbb{C}}\setminus\lbrace 1,\xi,...,\xi^{n-1} \rbrace$ with $\xi=e^{\frac{2\pi i}{n}}$, then the Weierstrass data of the n-noid are given by
\begin{align*}
g=z^{n-1},\quad \eta=\frac{dz}{(z^n-1)^2}.
\end{align*}
\end{example}
Now we are interested in discussing non-positively curved Ricci surfaces with positive genus and finitely many catenoidal ends. Our goal of this section is to prove that there exist Ricci surfaces  $M\simeq S_k\setminus\lbrace p_1,p_2,...,p_n \rbrace$ with $n$ catenoidal ends for $k,n>0$, where $S_k$ is a compact orientable surface of genus $k$. To achieve this, we will use the method introduced by Andrei Moroianu and Sergiu Moroianu as mentioned in remark \ref{rem4.2}. The first tool that we need is the following theorem proved by Gabriele Mondello and Dmitri Panov (see \cite{mondello2019spherical}, Theorem A).

\begin{theorem}\label{thm7.1}
Let $\chi\leq0$ be an even number and $v_1,...,v_n\in\mathbb{R}^*_{+}$ be such that 
\begin{align}\label{7.1}
\chi+\sum_{j=1}^{n}(v_j-1)>0,
\end{align}
then there exists a compact Riemann surface $S_k$ of genus $k=\frac{2-\chi}{2}\geq 1$, a set of distinct points $\lbrace p_1,p_2,...,p_n\rbrace\subset S_k$ and a metric $g_1$ of constant Gauss curvature 1 on $S_k$ such that $g_1$ has conical singularities of order $v_j-1$ at $p_j$.
\end{theorem}

In order to get a Ricci metric on $M\simeq S_k\setminus\lbrace p_1,p_2,...,p_n \rbrace$, we have to construct a proper flat metric. This is possible due to a result of Andrei Moroianu and Sergiu Moroianu (see \citep{moroianu2015ricci}, Lemma 6.1). 

\begin{lemma}\label{lem7.1}
Given a compact Riemann surface $S_k$ of genus $k$, let $p_1,p_2,...,p_n\in S_k$ and $\beta:\lbrace p_1,p_2,...,p_n\rbrace\rightarrow\mathbb{R}$ be a function which satisfies 
\begin{align}\label{7.2}
\sum_{j=1}^{n}(\beta(p_j)-1)=-\chi(S_k),
\end{align} 
then there exists a flat metric $g_0$ on $S_k\setminus\lbrace p_1,p_2,...,p_n \rbrace$ compatible with the complex structure of $S_k$ such that near each $p_i$, it has the form 
\begin{align}\label{7.3}
g_0=e^{2u}|z|^{2\beta(p_i)-2}|dz|^2
\end{align}
with some smooth function $u\in C^{\infty}(S_k,\mathbb{R})$.
\end{lemma}

Thanks to these two results, we are able to prove the existence of Ricci surfaces with arbitrary genus and arbitrary number of catenoidal ends.

\begin{theorem}\label{thm7.2}
For $k,n>0$, there exists a non-positively curved orientable Ricci surface $M$ of genus $k$ with $n$ cetenoidal ends.
\end{theorem}
\begin{proof}
Firstly, let us take $\chi=2-2k\leq0$, $v_1,...,v_n\in\mathbb{R}^*_{+}$ and $v_{n+1}=2(2k-2+n)+1>0$, then it is easy to verify that 
\begin{align*}
\chi+\sum_{j=1}^{n+1}(v_j-1)=\sum_{j=1}^{n}v_j+2k+n-2>0.
\end{align*} 
Hence by theorem \ref{thm7.1}, we will have a compact Riemann surface $S_k$ of genus $k$, a set of points $\lbrace p_1,p_2,...,p_{n+1}\rbrace\subset S_k$ and a metric $g_1$ of constant Gauss curvature 1 on $S_k$ with conical singularities of order $v_j-1$ at $p_j$. 

Secondly, we define a function $\beta:\lbrace p_1,p_2,...,p_{n+1}\rbrace\rightarrow\mathbb{R}$ as $\beta(p_i)=0$ for $i=1,2,...,n$ and $\beta(p_{n+1})=2k-1+n$. This function $\beta$ satisfies the equality (\ref{7.2}), thus from lemma \ref{lem7.1}, there is a flat metric $g_0$ on $S_k\setminus\lbrace p_1,p_2,...,p_{n+1} \rbrace$ which is conformal to $g_1$ and takes the form as in (\ref{7.3}) near every punctured point.

Both of $g_1$ and $g_0$ do not vanish on $S_k\setminus\lbrace p_1,p_2,...,p_{n+1} \rbrace$. Since $g_0$ is conformal to $g_1$, there is a positive function $V$ defined on $S_k\setminus\lbrace p_1,p_2,...,p_{n+1} \rbrace$ such that $g_1=Vg_0$. It follows that $g:=V^{-1}g_0$ is a Ricci metric on $S_k\setminus\lbrace p_1,p_2,...,p_{n+1} \rbrace$ (see \citep{moroianu2015ricci}, Lemma 2.3). Near the point $p_{n+1}$, it is known that the metric $g_1$ can be written as  
\begin{align}\label{7.5}
g_1=\frac{4v_{n+1}^2|z|^{4(2k-2+n)}|dz|^2}{(1+|z|^{4(2k-2+n)+2})^2}
\end{align}
for a suitable complex coordinate $z$. From lemma \ref{lem7.1}, the flat metric $g_0$ has also a local expression around $p_{n+1}$ which is 
\begin{align}\label{7.4}
g_0=e^{2u}|z|^{2(2k-2+n)}|dz|^2
\end{align}
for some smooth function $u$. 
A direct computation shows that the Ricci metric $g$ is of the form
\begin{align}\label{7.6}
g=\frac{1}{4}v_{n+1}^{-2}e^{4u}(1+|z|^{4(2k-2+n)+2})^2|dz|^2
\end{align}
near $p_{n+1}$. Hence $g$ is actually well-defined and smooth at $p_{n+1}$, thus a Ricci metric on $S_k\setminus\lbrace p_1,p_2,...,p_n \rbrace$. In addition, near each $p_i\in\lbrace p_1,p_2,...,p_n\rbrace$, $g_0$ has the form 
\begin{align*}
g_0=e^{2u}|z|^{-2}|dz|^2.
\end{align*}
This means $p_1,p_2,...,p_n$ are catenoidal ends of this Ricci surface (see the discussion after definition \ref{def3.2}). Consequently, we have constructed an orientable Ricci surface of genus $k$ with exactly $n$ catenoidal ends.
\end{proof}

\begin{example}\label{ex7.2}
Wayne Rossman and Katsunori Sato have constructed a genus 1 catenoid cousin which is a CMC-1 immersion into the hyperbolic 3-space $\mathcal{H}^3$ (see \cite{rossman1998constant}). This is an example of theorem \ref{thm7.2} in the case $k=1$ and $n=2$.
\end{example}

\bibliographystyle{plain}
\bibliography{Ricci}

\begin{thebibliography}{}

\bibitem[Bryant, 1987]{bryant1987surfaces}
Bryant, R. (1987).
\newblock Surfaces of mean curvature one in hyperbolic space.
\newblock {\em Ast{\'e}risque}, 154(155):321--347.

\bibitem[Calabi, 1968]{calabi1968quelques}
Calabi, E. (1968).
\newblock Quelques applications de l'analyse complexe aux surfaces d'aire
  minima.
\newblock {\em Topic in Complex Manifolds}, pages 59--81.

\bibitem[Daniel, 2021]{benoit2016survey}
Daniel, B. (2021).
\newblock A survey on minimal isometric immersions into {$\Bbb R^3$}, {$\Bbb
  S^2\times\Bbb R$} and {$\Bbb H^2\times\Bbb R$}.
\newblock In {\em Minimal Surfaces: Integrable Systems and Visualisation},
  pages 51--65. Springer Proceedings in Mathematics and Statistics.

\bibitem[H.~Blaine~Lawson, 1970]{lawson1970complete}
H.~Blaine~Lawson, J. (1970).
\newblock Complete minimal surfaces in $\mathbb{S}^3$.
\newblock {\em Annals of Mathematics}, 92(3):335--374.

\bibitem[Hoffman and Osserman, 1980]{hoffman1980geometry}
Hoffman, D.~A. and Osserman, R. (1980).
\newblock The geometry of the generalized gauss map.
\newblock {\em Memoirs of the American Mathematical Society}, 28(236):1--105.

\bibitem[Huber, 1958]{huber1958subharmonic}
Huber, A. (1958).
\newblock On subharmonic functions and differential geometry in the large.
\newblock {\em Commentarii Mathematici Helvetici}, 32(1):13--72.

\bibitem[Jorge and Meeks~III, 1983]{jorge1983topology}
Jorge, L.~P. and Meeks~III, W.~H. (1983).
\newblock The topology of complete minimal surfaces of finite total gaussian
  curvature.
\newblock {\em Topology}, 22(2):203--221.

\bibitem[Mondello and Panov, 2019]{mondello2019spherical}
Mondello, G. and Panov, D. (2019).
\newblock Spherical surfaces with conical points: systole inequality and moduli
  spaces with many connected components.
\newblock {\em Geometric and Functional Analysis}, 29(4):1110--1193.

\bibitem[Moroianu and Moroianu, 2015]{moroianu2015ricci}
Moroianu, A. and Moroianu, S. (2015).
\newblock Ricci surfaces.
\newblock {\em Annali della Scuola Normale Superiore di Pisa. Classe di
  scienze}, 14(4):1093--1118.

\bibitem[Osserman, 2013]{osserman2013survey}
Osserman, R. (2013).
\newblock {\em A survey of minimal surfaces}.
\newblock Courier Corporation.

\bibitem[Ricci-Curbastro, 1895]{ricci1895sulla}
Ricci-Curbastro, G. (1895).
\newblock Sulla teoria intrinseca delle superficie ed in ispecie di quelle di
  secondo grado, atti r.
\newblock {\em Ist. Ven. di Lett. ed Arti}, 6:445--488.

\bibitem[Rossman and Sato, 1998]{rossman1998constant}
Rossman, W. and Sato, K. (1998).
\newblock Constant mean curvature surfaces with two ends in hyperbolic space.
\newblock {\em Experimental Mathematics}, 7(2):101--119.

\bibitem[Troyanov, 1989]{troyanov1989metrics}
Troyanov, M. (1989).
\newblock Metrics of constant curvature on a sphere with two conical
  singularities.
\newblock In {\em Differential Geometry}, pages 296--306. Springer.

\bibitem[Umehara and Yamada, 1993]{umehara1993complete}
Umehara, M. and Yamada, K. (1993).
\newblock Complete surfaces of constant mean curvature-1 in the hyperbolic
  3-space.
\newblock {\em Annals of mathematics}, 137(3):611--638.

\bibitem[Umehara and Yamada, 2000]{Masaaki2000metrics}
Umehara, M. and Yamada, K. (2000).
\newblock Metrics of constant curvature 1 with three conical singularities on
  the $2 $-sphere.
\newblock {\em Illinois Journal of Mathematics}, 44(1):72--94.

\bibitem[Weber, 2005]{weber2005classical}
Weber, M. (2005).
\newblock Classical minimal surfaces in {E}uclidean space by examples:
  geometric and computational aspects of the {W}eierstrass representation.
\newblock {\em Global theory of minimal surfaces}, 2:19--63.

\end{thebibliography}

\end{document}